\documentclass[twoside, 12pt]{article}
\usepackage{a4wide, amsthm, amssymb, amscd, mathptmx}
\usepackage{amsfonts}
\usepackage[all]{xy}\CompileMatrices\SelectTips{cm}{12}

\theoremstyle{plain}
\newtheorem{Thm}{\sc Theorem}[section]
\newtheorem{Theorem}[Thm]{\sc Theorem}
\newtheorem{Corollary}[Thm]{\sc Corollary}
\newtheorem*{Corollary*}{\sc Corollary}
\newtheorem{Proposition}[Thm]{\sc Proposition}
\newtheorem*{Proposition*}{\sc Proposition}
\newtheorem{Lemma}[Thm]{\sc Lemma}

\theoremstyle{definition}

\theoremstyle{remark}

\newtheorem*{Example*}{Example}
\newtheorem*{Remark*}{Remark}

\renewcommand{\L}{{\cal L}}
\newcommand{\M}{{\cal M}}

\renewcommand{\O}{{\cal O}}

\newcommand{\db}{{{\rm D}}}
\newcommand{\Db}{{{\cal D}}}
\newcommand{\Diff}{{\cal D}}

\newcommand{\Spin}{\mathop{\rm Spin}}

\newcommand{\SO}{\mathop{\rm SO}}

\newcommand{\PP}{\mathbb{P}}

\newcommand{\ZZ}{\mathbb{Z}}

\renewcommand{\char}{\mathop{\rm char}}

\newcommand{\Coh}{\mathop{{\rm Coh}}}

\newcommand{\ext}{\mathop{{\rm Ext}}}

\newcommand{\Hom}{{\mathop{{\cal H}om}}}
\renewcommand{\hom}{{\mathop{\rm Hom}}}

\newcommand{\Proj}{ {\mathop{\rm Proj}}}

\newcommand{\rk}{\mathop{\rm rk}}

\newcommand{\SL}{\mathop{\rm SL}}

\newcommand{\ti}{\tilde}

\title{D-affinity and Frobenius morphism on quadrics}
\author{Adrian Langer}

\begin{document}

\maketitle

{\sc Address:}
Institute of Mathematics, Warsaw University, Ul.\ Banacha 2,
PL-02-097 Warszawa, Poland, e-mail: {\tt alan@mimuw.edu.pl}

\bigskip

{\sc Abstract:}
We compute decomposition of Frobenius push-forwards of line bundles on
quadrics into a direct sum of line bundles and spinor bundles. As an
application we show when the Frobenius push-forward gives a tilting
bundle and we apply it to study D-modules on quadrics.


\section*{Introduction}

Let $X$ be a smooth projective variety.  A coherent sheaf $E\in \Coh
X$ is called \emph{quasi-exceptional} if $\ext ^i(E,E) =0$ for all
$i>0$. A coherent sheaf $E\in \Coh X$ is called \emph{tilting} if it
is quasi-exceptional, $E$ Karoubian generates the derived category
$D^b(X)$ and the algebra $\hom _X(E, E)$ has finite global dimension.

Let $Q_n$ be an $n$-dimensional quadric defined over an algebraically
closed field $k$ of characteristic $p>2$. Assume for simplicity that
$n\ge 3$.  Let $F: Q_n\to Q_n$ be the absolute Frobenius morphism and
let $F^s$ be the composition of $s$ absolute Frobenius morphisms.

\begin{Theorem}\label{main}
Let $s$ be a positive integer. Then $F^s_*\O_{Q_n}$ is a tilting
bundle if and only if one of the following holds:
\begin{enumerate}
\item $s=1$ and $p>n$,
\item $s=2$, $n=4$ and $p=3$,
\item $s\ge 2$, $n$ is odd and $p\ge n$.
\end{enumerate}
\end{Theorem}

The above theorem is a summary of Corollaries \ref{t1}, \ref{t2} and
\ref{t3}. In case of $s=1$ and $n=3,4$ the theorem (see also Corollary
\ref{quasi-exc}) gives the main result of \cite{Sa1}. In case of $s=1$
A. Samokhin in \cite{Sa2} proved a related result for $p\gg 0$ but
using a completely different method.

In fact, we prove a much stronger result than stated above: we
determine the decomposition of Frobenius push-forwards of line bundles
on quadrics.  In case of projective spaces one can easily compute the
corresponding decomposition using the Horrocks splitting criterion
(see, e.g., \cite[Lemma 2.1]{Ra}; one can also perform another direct
computation as in \cite[Proposition 4.1]{HKR}). We use a similar
strategy in the case of quadrics although it is much more difficult in
carrying out as we need to prove some non-trivial vanishing and
non-vanishing theorems for cohomology of Frobenius pull-backs of
spinor bundles.

\medskip

Let $X$ be a smooth projective variety defined over an algebraically
closed field $k$ of characteristic $p>0$. Let $X^{(s)}$ denotes the
$s$-th Frobenius twist of $X$, i.e., $X$ with the $k$-structure
defined by the fiber product of $X$ over the $s$-th Frobenius morphism
on $k$. Then the sheaf $\Diff_X$ of $k$-algebras of differential
operators admits the so called \emph{$p$-filtration}
$$\Db ^1_X\subset \Db ^2_X\subset \dots \subset \Db^s_X=\Hom
_{X^{(s)}}(\O_X, \O_X)\subset \dots \subset {\Diff}_X=\lim
_{\mathop{\leftarrow}_{s}}\Db ^s_X.$$ Let us also set $\db^s_X=\hom
_{X^{(s)}}(\O_X, \O_X)$. By definition $\db^s_X$ is an algebra, finite
dimensional over $k$.  By \cite[Proposition 3.4]{HKR} the above
theorem implies the following corollary:

\begin{Corollary}
If one of the conditions 1, 2, 3 of Theorem \ref{main} holds then
there is a canonical triangulated equivalence between the bounded
derived category $D^b(\Coh (\Db _{Q_n}^s))$ of coherent $\Db
_{Q_n}^s$-modules and the bounded derived category $D^b(\db
_{Q_n}^s-{\mathrm{mod}})$ of $\db _{Q_n} ^s$-modules.
\end{Corollary}

This is an analogue of the derived localization theorem of Beilinson
and Bernstein. Note that contrary to some expectations this equivalence
does not hold in any characteristic if $n$ is even and $s\ge 3$.

Let us recall that a variety $X$ is called \emph{D-quasi-affine} if
any $\O_X$-quasi-coherent $\Diff _X$-module is $\Diff_X$-generated by
its global sections. $X$ is \emph{D-affine} if it is D-quasi-affine and
$H^i (\Diff _{Q_n})=0$ for $i>0$.

In characteristic zero Beilinson and Bernstein (see \cite{BB}) proved
that flag varieties for semisimple algebraic groups are D-affine. In
\cite{Ha} Haastert proved that in positive characteristic the
projective space and the full flag variety for $\SL _3$ are D-affine but
no other examples of smooth projective D-affine varieties are
known. We prove that all quadrics are D-quasi-affine (see Proposition
\ref{quasi-aff}). Together with Theorem \ref{main} this implies the
following corollary:

\begin{Corollary}
If $n$ is odd and $p\ge n$ then $Q_n$ is D-affine.
\end{Corollary}

Let us note that contrary to \cite[Proposition]{AK}, non-vanishing of
$H^i (\Db ^s_{Q_n})$ for even $n$ and $s\ge 3$ does not imply that
$H^i (\Diff _{Q_n})\ne 0$ (there is an error in the last two lines of
proof of \cite[Proposition]{AK} as the natural inclusions $\Db ^s\to
\Db^r$ do not agree with the maps from \cite[Lemma]{AK}).  In
particular, our results do not immediately imply that other quadrics
are not D-affine.

\medskip

We do not consider the case of characteristic $2$ although our methods
apply also in this case. We skip it as this case needs a separate,
quite long treatment. Let us just note that on the $3$-dimensional
quadric already $F^4_*\O_{Q_3}$ is not quasi-exceptional (but
$F^i_*\O_{Q_3}$ is quasi-exceptional if $i\le 3$).

\medskip

The structure of the paper is as follows. In Section 1 we give a few
interpretations of spinor bundles, compute some related cohomology and
show some dualities. In Section 2 we study Hilbert functions of some
algebras that occur to describe decompositions of Frobenius
push-forwards of line bundles on a quadric. In Section 3 we use these
results to prove some vanishing and non-vanishing theorems for
Frobenius pull-backs of spinor bundles. In Section 4 we use them for
computation of decompositions of Frobenius push-forwards of line
bundles on a quadric and we prove Theorem \ref{main}.

\section{Spinor bundles}

\subsection{Spinor bundles via representation theory}

Let $k$ be an algebraically closed field of characteristic $p>2$.  Let
us recall that $\Spin (n+2)$, $n\ge 1$, is a connected, semi-simple,
simply connected $k$-group, isogenous to $\SO (n+2)$.

We fix a maximal torus $T\subset \Spin (n+2)$.  If $n=2m+1$ or $n=2m$
then $\Spin (n+2)$ has $m+1$ simple roots $\alpha _1, \dots,
\alpha_{m+1}$.  Let us recall that a smooth $n$-dimensional quadric
$Q_n\subset \PP^N$, $N=n+1$, is a homogeneous space $\Spin
(n+2)/P(\alpha_1)$.

If $n=2m+1$ then the Dynkin diagram of $\Spin (n+2)$ is of type
$B_{m+1}$. If $n=2m$ then the Dynkin diagram is of type $A_1\times
A_1$ for $m=1$, $A_3$ for $m=2$ and $D_{m+1}$ if $m\ge 3$.

Let $\lambda_i\in X^*(T)$ be the fundamental weights defined by
$2\langle\lambda_i,\alpha_j \rangle/\langle\alpha_i,\alpha_j
\rangle=\delta_{ij}$.

For $n=2m$ there are two spin representations of the Levi
quotient of $P(\alpha_1)$ (which is of the same type as $\Spin (n)$)
with highest weight $\lambda_m$ and $\lambda_{m+1}$. They are of
dimension $2^{m-1}$.  If $n=2m+1$ then there is one spin
representation of the Levi quotient of $P(\alpha_1)$ with highest
weight $\lambda_{m+1}$. It is of dimension $2^{m}$.

Duals of vector bundles on $Q_n$ associated to the principal
$P(\alpha_1)$-bundle $\Spin (n+2)\to \Spin (n+2)/P(\alpha_1)$ via
these representations are called \emph{spinor bundles} and denoted by
$\Sigma$ if $n$ is odd and $\Sigma_{-}$ and $\Sigma_{+}$ if $n$ is
even.  The determinant of any spinor bundle on $Q_n$, $n\ge 3$ is
isomorphic to $\O _{Q_n}(-2^{[\frac{n-3}{2}]})$. Let us note the
following useful isomorphisms:
$$
\begin{array}{cl}
\Sigma ^*\simeq \Sigma (1) \quad &\hbox{if $n=2m+1$, } \\
{\Sigma}_{-} ^*\simeq {\Sigma}_{-} (1) \quad \hbox{ and }\quad {\Sigma}_{+}^*\simeq
{\Sigma_{+}}(1) \quad &\hbox{if $n=4m$, } \\
{\Sigma}_{-} ^*\simeq {\Sigma_{+}} (1) \quad \hbox{ and }\quad {\Sigma}_{+}^*\simeq
\Sigma_{-} (1) \quad &\hbox{if $n=4m+2$. } \\
\end{array}
$$ We leave it to the reader to verify these isomorphisms using the
well known computation of the centre of the Spin group.

Let us recall that a vector bundle $E$ on a smooth $n$-dimensional
hypersurface $X=(f=0)\subset \PP^{n+1}=\Proj\, S$ is called
\emph{arithmetically Cohen-Macaulay} (ACM for short), if it has no
intermediate cohomology, i.e., $H^i(X, E(t))=0$ for $0<i<n$ and all
$t$. A vector bundle $E$ is ACM if and only if the corresponding
graded $S/(f)$-module is maximal Cohen-Macaulay.

Let us also recall that a vector bundle is called strongly slope (semi)stable
if all its Frobenius pull-backs are slope (semi)stable.

\begin{Theorem} \label{coh-spin}
A spinor bundle $\Sigma$ on $Q_n\subset \PP^N$ is a strongly slope
stable ACM bundle. Moreover, $h^0(Q_n, \Sigma (t))=0$ for $t\le 0$ and
$h^0(Q_n, \Sigma (t))=2^{[\frac{N}{2}]}{{t+n-1}\choose {n}}$ for $t\ge
1$.
\end{Theorem}

\begin{proof}
Let us first note that spin representations are irreducible (this
follows from \cite[Part II, Corollary 5.6]{Ja}). Therefore by
\cite[Theorem 2.1]{Bi} spinor bundles are slope stable. Strong slope
stability follows in a standard way from inequality $\mu _{\max}(\Omega _{Q_n})<0$.

The fact that spinor bundles are ACM can be proven in the same way as
\cite[Theorem 2.3]{Ot}. The last part of the theorem will be obvious
later (use sequences (\ref{(i')}), (\ref{(ii')}) and (\ref{(iii')})).
\end{proof}

\subsection{Spinor bundles via matrix factorization}

Theorem \ref{coh-spin} implies that spinor bundles correspond to
irreducible maximal Cohen--Macaulay modules on an affine cone over the
quadric (or equivalently to indecomposable matrix factorizations of
the equation of the quadric). Below we give an explicit construction
of spinor bundles using matrix factorization.

As a special case of Kn\"orrer's periodicity theorem (see
\cite[Theorem 3.1]{Kn}) we get the following theorem:

\begin{Theorem} \label{ACM-Q}
Any ACM bundle on a smooth projective quadric defined over an
algebraically closed field of characteristic $p\ne 2$ is a direct sum
of line bundles and twisted spinor bundles.
\end{Theorem}

If $n\le 2$ then the Picard group is not generated by $\O_{Q_n}(1)$
and any ACM bundle on $Q_n$ is isomorphic to a direct sum of line
bundles $\O_{Q_n}(i)$ and spinor bundles twisted by some $\O
_{Q_n}(i)$. If $n\ge 3$ then any direct sum of line bundles and
twisted spinor bundles is ACM.

\medskip

Let us set $\varphi_0=\psi_0=(x_0)$ and let us define inductively
pairs of matrices
$$\varphi_{m+1}=\left(
\begin{array}{cc}
\varphi_{m} & x_{2m+1}I_{2^{m}\times 2^{m}}\\
x_{2m+2}I_{2^{m}\times 2^{m}}& -\psi _m\\
\end{array}
\right)$$
and
$$\psi_{m+1}=\left(
\begin{array}{cc}
\psi_{m} & x_{2m+1}I_{2^{m}\times 2^{m}}\\
x_{2m+2}I_{2^{m}\times 2^{m}}& -\varphi _m\\
\end{array}
\right) .$$
Then the pair $(\varphi_m,\psi_m)$ is a matrix factorization of
$x_0^2+x_1x_2+\dots +x_{2m-1}x_{2m}$.

Let us set $\varphi'_0=(x_1)$, $ \psi'_0=(x_2)$ and as above define
inductively pairs of matrices
$$\varphi'_{m+1}=\left(
\begin{array}{cc}
\varphi'_{m} & x_{2m+1}I_{2^{m}\times 2^{m}}\\
x_{2m+2}I_{2^{m}\times 2^{m}}& -\psi' _m\\
\end{array}
\right)$$
and
$$\psi'_{m+1}=\left(
\begin{array}{cc}
\psi_{m}' & x_{2m+1}I_{2^{m}\times 2^{m}}\\
x_{2m+2}I_{2^{m}\times 2^{m}}& -\varphi'_m\\
\end{array}
\right) .$$
Then the pair $(\varphi_m',\psi_m')$ is a matrix factorization of
$x_1x_2+\dots +x_{2m-1}x_{2m}$.

Let $i:Q_n \hookrightarrow \PP^N$ be the above defined embedding of a
quadric ($n=2m$ or $n=2m+1$). Then we have the following short exact
sequences of sheaves on $\PP^N$:
\begin{equation}
0\to \O_{\PP^N}^{2^{m+1}}(-1)\mathop{\to}^{\varphi_{m+1}}
\O_{\PP^N}^{2^{m+1}}\to i_*\Sigma (1)\to 0,\label{(i')}
\end{equation}
if $n=2m+1$, and
\begin{equation}
0\to \O_{\PP^N}^{2^{m}}(-1)\mathop{\to}^{\varphi'_m}
\O_{\PP^N}^{2^{m}}\to i_*\Sigma _{-}(1)\to 0,\label{(ii')}
\end{equation}
\begin{equation}
0\to \O_{\PP^N}^{2^{m}}(-1)\mathop{\to}^{\psi'_m}
\O_{\PP^N}^{2^{m}}\to i_*\Sigma _{+}(1)\to 0.\label{(iii')}
\end{equation}
if $n=2m$.

Using the above description we get the following short exact sequences
of vector bundles:
\begin{equation}
0\to \Sigma \to \O_{Q_n}^{2^{m+1}}\to \Sigma (1)\to 0,\label{(i)}
\end{equation}
if $n=2m+1$, and
\begin{equation}
0\to \Sigma_{-}\to \O_{Q_n}^{2^{m}}\to \Sigma _{+}(1)\to 0,\label{(ii)}
\end{equation}
\begin{equation}
0\to \Sigma _{+}\to \O_{Q_n}^{2^{m}}\to \Sigma _{-}(1)\to 0,\label{(iii)}
\end{equation}
if $n=2m$. It should be noted that the above explicit presentations
allow for computer calculations of cohomology groups of Frobenius pull
backs of spinor bundles.

\begin{Lemma} \label{tensor-spin}
For any spinor bundles $\Sigma_1, \Sigma _2$ on $Q_n$
we have $H^1(Q_n, \Sigma_1 \otimes \Sigma _2(t))=0$ if $t\ne 0$.
Moreover,
$$
h^1(Q_n, \Sigma_1 \otimes \Sigma _2)=\left\{
\begin{array}{cl}
1& \quad\hbox{if $n=2m+1$,}\\
1& \quad\hbox{if $n=4m$ and $\Sigma_1\not\simeq \Sigma_2$,}\\
0& \quad\hbox{if $n=4m$ and $\Sigma_1\simeq \Sigma_2$,}\\
1& \quad\hbox{if $n=4m+2$ and $\Sigma_1\simeq \Sigma_2$,}\\
0& \quad\hbox{if $n=4m+2$ and $\Sigma_1\not \simeq \Sigma_2$.}\\
\end{array}\right.
$$ In particular, for any $0<i<n$ there exist spinor bundles
$\Sigma_1, \Sigma_2$ such that $\ext ^i(\Sigma_1(i), \Sigma_2)\ne 0$.
\end{Lemma}

\begin{proof}
As spinor bundles on a quadric are strongly stable we see that $\hom
(\Sigma_1, \Sigma_2(t))=0$ if $t<0$ or if $t=0$ and $\Sigma_1$ and
$\Sigma_2$ are not isomorphic. This remark, together with sequences
(\ref{(i)}), (\ref{(ii)}), (\ref{(iii)}), imply the second part of the
lemma.

To prove the first assertion note that by Lemma \ref{2-spin-duality}
for $s=0$ there exist spinor bundles $\ti \Sigma_1, \ti \Sigma _2$
such that
$$H^1(Q_n, \Sigma_1 \otimes \Sigma _2(t))\simeq H^1(Q_n, \ti \Sigma_1
\otimes \ti \Sigma _2(-t))^*.$$ So it is sufficient to note that by
Theorem \ref{coh-spin} sequences (\ref{(i)}), (\ref{(ii)}),
(\ref{(iii)}) imply that $H^1(\Sigma_1 \otimes \Sigma _2(t))$ is a
quotient of $H^0(\Sigma_1' \otimes \Sigma _2(t+1))$ for some spinor
bundle $\Sigma_1'$. This last cohomology group vanishes for $t<0$ as
we can write it as $\hom (\Sigma '', \Sigma_2(t))$ for some spinor
bundle $\Sigma''$.

The last part of the lemma follows from isomorphisms
$${\ext} ^i(\Sigma_1(i), \Sigma_2)\simeq H^i(\Sigma_1^*\otimes
\Sigma_2(-i))\simeq H^1(\Sigma_1'\otimes \Sigma_2)$$ for some spinor
bundle $\Sigma_1'$.
\end{proof}

\begin{Corollary} \label{dir-sum1}
Let $E$ be an ACM bundle on $Q_n$ and let $0<i<n$ be a fixed integer.
If for all spinor bundles $\Sigma$ on $Q_n$ we have
$H^i(Q_n, E\otimes \Sigma (t))=0$ for all $t\in \ZZ$ then $E$ is a direct sum
of line bundles.
\end{Corollary}

\begin{proof}
Using sequences (\ref{(i)}), (\ref{(ii)}), (\ref{(iii)}) and Theorem
\ref{coh-spin} we see that $H^i(E\otimes \Sigma (t))=H^1(E\otimes
\Sigma ' (t+i-1))$ for some spinor bundle $\Sigma'$.  Then the
required assertion follows from Theorem \ref{ACM-Q} and Lemma
\ref{tensor-spin}.
\end{proof}

\subsection{Some dualities}

For a non-negative integer $s$ let us set $q=p^s$ and
$d_{n,s}=(n-1)\frac{q-1}{2}$.

\begin{Lemma}\label{spin-duality}
Let $\Sigma$ be a spinor bundle on $Q_n$. Then for all $0<i<n$ there
exists some spinor bundle $\ti \Sigma$ such that for all integers
$j$ we have the following duality
$$H^i(Q_n, (F^s)^*\Sigma (d_{n,s}-(i-1)q+j)) \simeq H^i(Q_n,
(F^s)^*\ti\Sigma (d_{n,s}-(i-1)q-j-1))^*.$$
\end{Lemma}

\begin{proof}
Let us first prove the lemma for $i=1$.  Taking Frobenius pull backs of
sequences (\ref{(i)}), (\ref{(ii)}), (\ref{(iii)}) and twisting by
$\O_{Q_n}(j)$ we get the following isomorphisms
$$H^i(Q_n, (F^s)^*\Sigma (q+j))\simeq H^{i+1}(Q_n, (F^s)^*\Sigma _1 (j)),$$
for some spinor bundle $\Sigma _1$, $0<i<n-1$ and all integers $j$.
Hence we get
$$H^1(Q_n, (F^s)^*\Sigma (j))\simeq H^{2}(Q_n, (F^s)^*\Sigma_1(j-q))\simeq
\dots \simeq H^{n-1}(Q_n, (F^s)^*\Sigma_{n-2}(j-(n-2)q))$$ for some spinor
bundles $\Sigma _1, \dots , \Sigma _{n-2}$. Now  using the Serre duality we have
$$H^{n-1}(Q_n, (F^s)^*\Sigma_{n-2}(j-(n-2)q))\simeq H^1(Q_n, (F^s)^*\ti \Sigma
((n-1)(q-1)-j-1))^*$$ for some spinor bundle $\ti \Sigma$, which
proves the lemma for $i=1$.

In general, there exist some spinor bundles $\Sigma_1, \ti
\Sigma_1$ and $\ti \Sigma$ such that
$$\begin{array}{rl}
H^i(Q_n, (F^s)^*\Sigma (j))\simeq & H^1(Q_n, (F^s)^*\Sigma _1(j+(i-1)q))
\simeq H^1(Q_n, (F^s)^*\ti\Sigma_1 (2d_{n,s}-(i-1)q-j-1))^*\\
\simeq & H^i(Q_n, (F^s)^*\ti\Sigma (2d_{n,s}-2(i-1)q-j-1))^*.
\end{array}$$
\end{proof}

From the proof of the lemma it is clear that we can easily determine
dependence of $\ti \Sigma$ on $\Sigma$ but we need to consider some
cases depending on $n \pmod{4}$.  More precisely, $\ti \Sigma =\Sigma$
if $n$ is odd or $n$ is divisible by $4$ and $\ti \Sigma$ is the
opposite spinor bundle otherwise (at least for $i=1$).
\medskip

Similarly as above we have the following duality:

\begin{Lemma}\label{2-spin-duality}
Let $\Sigma_1$ and $\Sigma _2$ be spinor bundles on $Q_n$ (possibly
equal).  Then for all $0<i<n$ there exists some spinor bundles
$\ti \Sigma_1$ and $\ti \Sigma _2$ such that for all integers $j$ we
have the following duality
$$H^i(Q_n, \Sigma_1\otimes (F^s)^*\Sigma_2 (d_{n,s}-(i-1)q+j)) \simeq
H^i(Q_n,\ti \Sigma_1\otimes (F^s)^*\ti\Sigma_2
(d_{n,s}-(i-1)q-j))^*.$$
\end{Lemma}

\begin{proof}
The proof of the lemma is similar to that of Lemma \ref{spin-duality},
and so we just sketch it for $i=1$.  We can easily show the following
isomorphisms
$$H^1(Q_n, \Sigma_1\otimes (F^s)^*\Sigma_2 (j))\simeq \dots \simeq
H^{n-1}(Q_n, \Sigma_1'\otimes (F^s)^*\Sigma_2' (j-(n-2)q))$$ for some
spinor bundles. Finally, using the Serre duality we have
$$H^{n-1}(Q_n, \Sigma_1'\otimes (F^s)^*\Sigma_2' (j-(n-2)q))\simeq
H^1(Q_n, \ti \Sigma_1\otimes (F^s)^*\ti\Sigma_2 ((n-1)(q-1)-j))^*.$$
\end{proof}

Similarly as above one can easily find $\ti \Sigma_1$ and $\ti
\Sigma_2$ corresponding to $\Sigma_1$ and $\Sigma _2$.

\subsection{D-quasi-affinity of quadrics}

Let $L$ be a line bundle on a smooth projective variety $X$ and let
$\Diff _X(L)$ be the $\O_X$-bimodule of differential operators from
$L$ to $L$ (see \cite[1.1]{Ha}).  $X$ is called
\emph{D(L)-quasi-affine} if any $\O_X$-quasi-coherent $\Diff
_X(L)$-module is $\Diff_X(L)$-generated by its global sections.  We
say that $X$ is \emph{D(L)-affine} if it is D(L)-quasi-affine and for
any $\O_X$-quasi-coherent $\Diff _X(L)$-module $\M$ we have $H^i
(X,\M)=0$ for $i>0$.

\begin{Proposition} \label{quasi-aff}
Let $j$ be a non-negative integer. Then any smooth projective quadric
$Q_n$ is $D(\O_{Q_n}(j))$-quasi-affine.
\end{Proposition}

\begin{proof}
The proof is analogous to the proof of \cite[Satz 3.1]{Ha}.  By
\cite[Proposition 2.3.3]{Ha} it is sufficient to show that for any
integer $t$ the module $\Diff (\O_{Q_n}(j))\otimes \O_{Q_n}(-t)$ is
$\Diff_X(\O_{Q_n}(j))$-generated by its global sections.  Therefore by
\cite[Proposition 2.3.4]{Ha} it is sufficient to show that for any
positive integer $t$ there exists $s_0$ such that for every $s\ge s_0$
$(F^s_*(\O_{Q_n}(j+t)))^*$ is globally generated as an
$\O_{Q_n}$-module.  Let us take as $s_0$ any integer such that
$p^{s_0}>j+t$. By Theorem \ref{ACM-Q} we can write
$$F^s_*(\O_{Q_n}(j+t))\simeq \bigoplus \O _{Q_n}(a_i)\oplus \bigoplus
\Sigma_i(b_i)$$ for some spinor bundles $\Sigma_i$ and some integers
$a_i$ and $b_i$. Hence we need to show that all $a_i$ and $b_i$ are
non-positive.

Note that $\O _{Q_n}(a_i)\hookrightarrow F^s_*(\O_{Q_n}(j+t))$ gives
rise to a non-zero map $(F^s)^*(\O _{Q_n}(a_i))=\O _{Q_n}(p^s a_i)\to
\O_{Q_n}(j+t)\subset \O_{Q_n}(p^{s_0}-1)$, so $a_i\le 0$.  Similarly,
$\Sigma _i(b_i)\hookrightarrow F^s_*(\O_{Q_n}(j+t))$ gives rise to a
non-zero map $(F^s)^*\Sigma_i(p^s b_i)\to \O_{Q_n}(j+t)\subset
\O_{Q_n}(p^{s_0}-1)$. This gives rise to a section of
$(F^s)^*\Sigma_i^*(p^{s_0}-1-p^sb_i)$. Now it is sufficient to show
that $H^0((F^s)^*\Sigma_i^*(-1))=0$ as then $b_i\le 0$.  But this can
be easily shown by restricting to quadrics of lower dimension and
induction on the dimension $n$ (cf.~the proof of Theorem
\ref{vanishing}).
\end{proof}

\section{Hilbert functions of some algebras}

In this section we study Hilbert functions of some finite dimensional
algebras in positive characteristic. Their geometric meaning will
become clear in Sections 3 and 4 (see the proof of Corollary
\ref{dir-sum-lb} and Theorem \ref{decomposition}).

\begin{Proposition} \label{diff-new}
Let $k$ be a field of characteristic $p>2$. Let  $0\le e<p$  be an
integer. Then for any $d\le (N+1)\cdot \frac{p-1}{2}-e$ the $d$th
grading of the ideal $((x_0^p,\dots,
x_N^p):(\sum_{i=0}^Nx_i^2)^{e})$ of $k[x_0,\dots ,x_N]$ is
contained in $(x_0^p,\dots, x_N^p, \sum_{i=0}^Nx_i^2)_d$.
\end{Proposition}

\begin{proof}
We need to prove that if for some homogeneous polynomial $h\in
k[x_0,\dots ,x_N]_d$ with $d+e\le (N+1)\cdot \frac{p-1}{2}$ we
have
\begin{equation}
(\sum_{i=0}^Nx_i^2)^{e} h=\sum_{i=0}^N g_i x_i^p \label{eq1}
\end{equation}
for some homogeneous polynomials $g_i\in k[x_0,\dots ,x_N]$ then
$$h\in (x_0^p,\dots, x_N^p, \sum_{i=0}^Nx_i^2).$$
We can assume that $h=h_0+h_1x_0+...+h_{p-1}x_0^{p-1}$, where
$h_0, ... , h_{p-1}$ are polynomials in $x_1,..., x_N$. Moreover,
we can also assume that all $h_i$s  are of degree less than $p$ in
each variable. Let us set $y=\sum_{i=1}^Nx_i^2$. To simplify
exposition we divide the proof into two cases: first we deal with
the ``even'' part of $h$ and later we deal with the ``odd'' part
of $h$.

\medskip

Let us set
$$
W_{2j}={e\choose j}h_0+{e\choose j-1}y h_2+\dots + {e\choose 0}y^j
h_{2j}
$$
for $j=0,\dots ,\frac{p-1}{2}$. Note that in $\sum_{i=0}^N g_i
x_i^p$ treated as a polynomial in $x_0$, coefficients at $x_0^i$
for $i<p$ are in the ideal $(x_1^p,..., x_N^p)$. Comparing
coefficients in (\ref{eq1}) at even powers of $x_0$ we get
therefore the  following equalities:
\begin{equation}
\label{eq2} y^{e-j}W_{2j}  \equiv 0\pmod{(x_1^p,\dots ,x_N^p)}
\end{equation}
for $j=0,\dots, \frac{p-1}{2}$ (note that if $j>e$ then all the
terms of  $W_{2j}$ are divisible by $y^{j-e}$ and we still have a
polynomial on the left hand side of the above equality).

\begin{Lemma} \label{h_j}
We have
$$y^jh_{2j}=(-1)^{j}{e+j-1\choose e-1}W_0+ (-1)^{j-1}{e+j-2\choose e-1}W_2+\dots
+ (-1)^{0}{e-1\choose e-1}W_{2j}.
$$
\end{Lemma}

\begin{proof}
Let $a_0, ..., a_{p-1}$ be arbitrary elements in a fixed ring $R$.
Consider the following equality of formal power series in
$R[[x]]$:
$$a_0+a_1x +\dots +a_{p-1}x^{p-1}=(1+x^2)^{-e}\left( (1+x^2)^e(a_0+a_1x +\dots +a_{p-1}x^{p-1})\right).$$
Now use expansion
$$(1+x^2)^{-e}=\sum _{j\ge 0} {-e \choose j}x^{2j}=\sum _{j\ge 0} (-1)^j{e +j-1 \choose j}x^{2j}$$
and multiply this by expansion of the product
$(1+x^2)^e(a_0+a_1x_0 +\dots +a_{p-1}x_0^{p-1})$. Then comparing
appropriate coefficients of the product we compute $a_j$. The
assertion in the lemma is a special case of such equality when
comparing coefficients at even powers of $x$.
\end{proof}

\begin{Lemma} \label{combination}
If $j\ge e-\frac{p+1}{2}$ then $ W_{2j}$  is a linear combination
of  $W_{{2j+2}}$, \dots , $W_{p-1}$ and
$y^{\frac{p-1}{2}-j}h_{p-1-2j}$, \dots , $y^{\frac{p-1}{2}}
h_{p-1}$.
\end{Lemma}

\begin{proof}
Set $m=\frac{p-1}{2}-j$. By Lemma \ref{h_j} we have
$$\sum _{k\le j } (-1)^{k} {e+s-1-k \choose e-1}W_{2k}
=  y^{s} h_{s}-\sum _{k>j } (-1)^{k} {e +s-1-k\choose e-1}W_{2k}$$
for $s=m, ..., m+j$. The determinant of this linear system of
equations is up to sign equal to
$$\left|
\begin{array}{ccc}
{e-1+m\choose {e-1} } & ... & {e-1+m-j\choose {e-1}}\\
...&...& ...\\
{e-1+m+j\choose {e-1}}& ... & {e-1+m \choose {e-1}}\\
\end{array} \right|= \left|
\begin{array}{ccc}
{e-1+m-j\choose {e-1-j} } & ... & {e-1+m-j\choose {e-1}}\\
...&...& ...\\
{e-1+m\choose {e-1-j}}& ... & {e-1+m \choose {e-1}}\\
\end{array} \right|= \left|
\begin{array}{ccc}
{e-1+m-j\choose {e-1-j} } & ... & {e-1+m-j\choose {e-1}}\\
...&...& ...\\
{e-1+m-j\choose {e-1-2j}}& ... & {e-1+m-j \choose {e-1-j}}\\
\end{array} \right|$$
where the first equality follows by inductive subtracting columns,
and the second equality follows by inductive subtracting rows. The
last determinant can be computed and by \cite{HT} it is equal to
$$\prod _{i=0}^{j}
\frac{ (e-j+m-1+i)\,! \, i\,!}{ (m-1+i)\,!   \, (e-i)\,!}.$$ It is
non-zero when $e+m-1\le p-1$, which is equivalent to $j\ge
e-\frac{p+1}{2}$. In this case we can solve the above linear
system of equations obtaining the required assertion.
\end{proof}

\begin{Lemma} \label{bound2}
If  $e-\frac{p-1}{2}\le j \le \frac{p-1}{2}$ then there exists a
polynomial $W'_{2j}$ such that $W_{2j}-y^{p-e+j}W'_{2j}$ lies in
the ideal $(x_1^p,\dots ,x_N^p)$.
\end{Lemma}

\begin{proof}
Assume that the assertion of the lemma is false and choose the
largest $j \le \frac{p-1}{2}$ for which it fails. If $j\ge e$ then
we get a contradiction with (\ref{eq2}) as we know that
$W_{2j}\equiv 0\pmod{(x_1^p,\dots ,x_N^p)}$. Therefore $j<e$. By
assumption $j\ge e-\frac{p-1}{2}$ and for every $m>j$ there exists
a polynomial $W'_{2m}$ such that $W_{2m}\equiv
y^{p-e+m}W'_{2m}\pmod{(x_1^p,\dots ,x_N^p)}$. By Lemma
\ref{combination} there exist constants
$b_0,...,b_{\frac{p-1}{2}}$ such that
\begin{equation} \label{eq3}
W_{2j} = b_{0} y^{\frac{p-1}{2}-j}h_{p-1-2j}+ \dots+ b_{j}
y^{\frac{p-1}{2}} h_{p-1} + b_{j+1} W_{{2j+2}} +\dots +
b_{\frac{p-1}{2}}W_{p-1}.
\end{equation}
If $2j\le e- \frac{p-1}{2}$ then $\frac{p-1}{2}-j\ge p-e+j$ and
for the polynomial
$$W_{2j} '= b_0
y^{\frac{p-1}{2}-j-(p-e+j)}h_{p-1-2j}+ \dots+ b_j
y^{\frac{p-1}{2}-(p-e+j)} h_{p-1}+b_{j+1} yW'_{{2j+2}} +\dots +
b_{\frac{p-1}{2}}y^{\frac{p-1}{2}-j} W'_{p-1}.
$$
we have $W_{2j}\equiv y^{p-e+j}W'_{2j}\pmod{(x_1^p,\dots
,x_N^p)}$.  Therefore $2j> e- \frac{p-1}{2}$. We have
$$ y^{e-j} W_{2j} \equiv
0\pmod{(x_1^p,\dots ,x_N^p)}.
$$
and from equations (\ref{eq2}) we see that $ y^{e-j} W_{2m} \equiv
0\pmod{(x_1^p,\dots ,x_N^p)}$ for $m>j$. Hence by (\ref{eq3}) we
have
$$
y^{e+\frac{p-1}{2}-2j} (b_0 h_{p-1-2j}+ \dots+ b_j y^{j}
h_{p-1})\equiv 0\pmod{(x_1^p,\dots ,x_N^p)}.$$ But we know that
$0\le e+\frac{p-1}{2} -2j<p$ and
$$\left( e+\frac{p-1}{2} -2j\right)+\deg h_{p-1-2j}=d+e-\frac{p-1}{2}\le N\frac{p-1}{2},$$
so we can apply the induction assumption. Therefore there exists a
polynomial $P$ such that
$$b_0h_{p-1-2j}+ \dots+ b_j
y^{j} h_{p-1}\equiv y^{p-e+j-(\frac{p-1}{2} -j)} P
\pmod{(x_1^p,\dots ,x_N^p)} .
$$
If we multiply this equality by $y^{\frac{p-1}{2}-j}$ and use
equation (\ref{eq3}) and we see that
$$W_{2j}\equiv y^{p-e+j}\left(P+b_{j+1} yW'_{2j+2}+\dots + b_{\frac{p-1}{2}}y^{\frac{p-1}{2}-j}W_{p-1}'\right)
\pmod{(x_1^p,\dots ,x_N^p)}.
$$
This gives a contradiction with our choice of $j$.
\end{proof}

\begin{Lemma}
$$\sum _{j=0}^{\frac{p-1}{2}} (-1)^jy^jh_{2j}={{e+\frac{p-1}{2}}\choose e}W_0+
{{e+\frac{p-1}{2} -1}\choose e}W_2+\dots + {e\choose e}W_{p-1}.$$
\end{Lemma}

\begin{proof}
Use  Lemma \ref{h_j} and equality
$${n\choose n}+{n+1\choose n}+\dots +{n+m\choose n}={n+m+1\choose n+1}.$$
\end{proof}

Since ${{e+\frac{p-1}{2} -j}\choose e}$ vanishes if
$j<e-\frac{p-1}{2}$ we have
$$
h_0+h_2x_0^2+...+h_{p-1}x_0^{p-1}\equiv \sum
_{j=0}^{\frac{p-1}{2}} (-1)^jy^jh_{2j} = \sum _{j \ge
e-\frac{p-1}{2}} {{e+\frac{p-1}{2} -j}\choose e}W_{2j}
\pmod{(x_1^p,\dots x_N^p, \sum_{i=0}^N x_i^2)}.
$$
Now by Lemma \ref{bound2} for any $j \ge e-\frac{p-1}{2}$ there
exists $W'_{2j}$ such that $W_{2j}\equiv
y^{p-e+j}W'_{2j}\pmod{(x_1^p,\dots ,x_N^p)}$. But $p-e+j\ge
\frac{p+1}{2}$ and
$$y^{\frac{p+1}{2}}\equiv \pm \, x_0^{p+1}\equiv 0 \pmod{(x_0^p,\dots x_N^p, \sum_{i=0}^N x_i^2)}.$$
Therefore $ h^{\rm even}=h_0+h_2x_0^2+...+h_{p-1}x_0^{p-1} $
belongs to the ideal $(x_0^p,\dots x_N^p, \sum_{i=0}^N x_i^2)$.

\bigskip
In a similar way we deal with the remaining part of $h$. First,
let us set
$$
W_{2j+1}={e\choose j}h_1+{e\choose j-1}y h_3+\dots + {e\choose
0}y^j h_{2j+1}
$$
for $j=0,\dots ,\frac{p-3}{2}$. Comparing coefficients in
(\ref{eq1}) at odd powers of $x_0$ we get
\begin{equation}
\label{eq2'} y^{e-j}W_{2j+1}  \equiv 0\pmod{(x_1^p,\dots ,x_N^p)}.
\end{equation}
for $j=0,\dots ,\frac{p-3}{2}$.

\begin{Lemma} \label{h_j-odd}
We have
$$y^jh_{2j+1}=(-1)^{j}{e+j-1\choose e-1}W_1+ (-1)^{j-1}{e+j-2\choose e-1}W_3+\dots
+ (-1)^{0}{e-1\choose e-1}W_{2j+1}.
$$
\end{Lemma}

\begin{Lemma} \label{combination-odd}
If $j\ge e-\frac{p+1}{2}$ then $ W_{2j+1}$  is a linear
combination of  $W_{{2j+3}}$, \dots , $W_{p-2}$ and
$y^{\frac{p-3}{2}-j}h_{p-2-2j}$, \dots , $y^{\frac{p-3}{2}}
h_{p-2}$.
\end{Lemma}

\begin{Lemma} \label{bound2-odd}
If  $e-\frac{p+1}{2}\le j \le \frac{p-3}{2}$ then there exists a
polynomial $W'_{2j+1}$ such that $W_{2j+1}-y^{p-e+j}W'_{2j+1}$
lies in the ideal $(x_1^p,\dots ,x_N^p)$.
\end{Lemma}

\begin{Lemma}
$$\sum _{j=0}^{\frac{p-3}{2}} (-1)^jy^jh_{2j+1}={{e+\frac{p-3}{2}}\choose e}W_1+
{{e+\frac{p-3}{2} -1}\choose e}W_3+\dots + {e\choose e}W_{p-2}.$$
\end{Lemma}

Since ${{e+\frac{p-3}{2} -j}\choose e}$ vanishes if
$j<e-\frac{p+3}{2}$, we have
$$
h_1+h_3x_0^2+...+h_{p-2}x_0^{p-3}\equiv \sum
_{j=0}^{\frac{p-3}{2}} (-1)^jy^{j}h_{2j+1} = \sum _{j \ge
e-\frac{p+1}{2}} {{e+\frac{p-3}{2} -j}\choose e}W_{2j+1}
\pmod{(x_1^p,\dots x_N^p, \sum_{i=0}^N x_i^2)}.
$$
Now by Lemma \ref{bound2-odd} for any $j \ge e-\frac{p+1}{2}$
there exists $W'_{2j+1}$ such that $W_{2j+1}\equiv
y^{p-e+j}W'_{2j+1}\pmod{(x_1^p,\dots ,x_N^p)}$. But $p-e+j\ge
\frac{p-1}{2}$ and
$$y^{\frac{p-1}{2}}\equiv \pm \, x_0^{p-1} \pmod{(x_0^p,\dots x_N^p, \sum_{i=0}^N x_i^2)}.$$
Therefore $ h^{\rm odd}=h_1x_0+h_3x_0^3+...+h_{p-2}x_0^{p-2} $
belongs to the ideal $(x_0^p,\dots x_N^p, \sum_{i=0}^N x_i^2)$.

Since $h=h^{\rm even}+h^{\rm odd}$, this finishes proof of
the proposition.
\end{proof}

\medskip

The following corollary of Proposition \ref{diff-new} is the main
step in our proof of Theorem \ref{vanishing}:

\begin{Proposition} \label{diff}
Let $k$ be a field of characteristic $p>2$. Let  $0\le e<p$  be an
integer. Then for any $d\le (N+1)\cdot \frac{p-1}{2}-e$ we have
$$((x_0^p,\dots, x_N^p):(\sum_{i=0}^Nx_i^2)^{e})_d=(x_0^p,\dots,
x_N^p,(\sum_{i=0}^Nx_i^2)^{p-e})_d.
$$
in $k[x_0,\dots ,x_N]$.
\end{Proposition}

\begin{proof}
We need to prove that for every homogeneous degree $d$ (with
$d+e\le (N+1)\cdot \frac{p-1}{2}$) polynomial $h$ such that
$$(\sum_{i=0}^Nx_i^2)^{e} h\equiv 0 \pmod{(x_0^p,\dots x_N^p)}$$
we have $h\in(x_0^p,\dots, x_N^p,(\sum_{i=0}^Nx_i^2)^{p-e})_d. $

If $e=0$ then the assertion is trivial so we can assume that $e\ge
1$.

The proof is by induction on $N$. For $N=0$ we have $2(d+e)\le
p-1$. But if $h\ne 0$ then counting degrees we get $d+2e\ge p$, a
contradiction. Now assume that the assertion holds for polynomials
in $N$ variables (i.e., for $N-1$).

If $e=p-1$ then the required assertion follows from Proposition
\ref{diff-new}. Otherwise, by Proposition \ref{diff-new} there
exists a polynomial $h'$ such that
$$h\equiv (\sum_{i=0}^Nx_i^2)h'\pmod{(x_0^p,\dots x_N^p)}.$$
Therefore
$$(\sum_{i=0}^Nx_i^2)^{e+1} h'\equiv (\sum_{i=0}^Nx_i^2)^{e} h\equiv 0
\pmod{(x_0^p,\dots x_N^p)}$$ and now we can again apply
Proposition \ref{diff-new}, since $e+1+\deg h'=e+\deg h\le
(N+1)\cdot \frac{p-1}{2}$. Continuing in this way till new $e$
becomes $p-1$, we get the required assertion.
\end{proof}

\medskip

Let us set $S=k[x_0, \dots , x_{N}]$, where $\char k=p$ and $N=n+1$.
Let $q=p^s$ for some non-negative integer $s$. Set
$I=(x_0^2+\dots+x_{N}^2,x_1^q,\dots ,x_{N}^q)$ and $A=S/I$. Let us
recall that
$$
\left(\frac{1-t^q}{1-t}\right) ^N=\sum _{i=0}^{\infty}\alpha_{i,N} t^i,$$
where $$\alpha_{i,N}=\sum _{j=0}^N(-1)^j
{{N}\choose {j}} {{i-jq+N-1}\choose {N-1}} ,$$
where we set ${a\choose b}=0$ if $a<b$ or $b<0$.
The Hilbert series of $A$ can be computed as
$$h_A(t)=\sum \dim A_i\, t^i=h_S(t) (1-t^2)(1-t^q)^N=(1+t)\left(
\frac{1-t^q}{1-t}\right) ^N.$$
Therefore
$$\dim A_i=\alpha_{i,N}+\alpha_{i-1,N}.$$ Note that $\dim_k A=2q^N$
(e.g., by the Bezout's theorem). The following lemma will be used in
the proof of Lemma \ref{sum-B}.

\begin{Lemma} \label{p^n}
For every integer $i$ we have
$$\sum _{j\in \ZZ}\dim A_{i+jq}=2q^n.$$
\end{Lemma}

\begin{proof} By definition
$$(1+t+\dots +t^{q-1})\sum _{i=0}^{\infty}\alpha_{i,n}t^i=\sum _{i=0}^{\infty}\alpha _{i,N}t^i.$$
This gives $\alpha_{i,N}=\alpha_{i,n}+\dots +\alpha_{i-q+1,n}$.
Hence
$$\sum _{j\in \ZZ}\alpha_{i+jq,N}=\sum _{j\in \ZZ}\alpha _{j,n}=q^n.$$
Now the lemma follows from equality $\dim A_l=\alpha_{l,N}+\alpha_{l-1,N}.$
\end{proof}

\medskip

Let us set $d_{N,s}= \frac{1}{2}n(q-1)$ and
$$\gamma_N(i)=\frac{1}{2^N} \sum_{j\in \ZZ}(-1)^j\dim A_{d_{N,s}+i+jq}$$
for $i\in \ZZ$.  In principle, we could write the formulas for
$\gamma_N(i)$ using formulas for $\dim A_j$, but the obtained formulas
are rather useless and we need different formulas. Let us first define
inductively some sequences of numbers and functions.  Set $w_{0}=1$
and assume we have defined integers $w_0, ..., w_k$. Then we set
$$F_{k}(i)=\sum _{j=0}^{k}(-1)^j w_{k-j} {{i+j} \choose {2j+1}}$$
and
$$w_{k+1}=\sum _{j=0}^{k}(-1)^j w_{k-j} {{\frac{q+1}{2}+j} \choose
{2j+2}}.$$ Similarly, set $u_{0}=0$ and assume we have defined $u_0,
...,u_k$.  Then we set
$$G_{k}(i)=(-1)^k {{i+k} \choose {2k}}+ \sum _{j=0}^{k}(-1)^j u_{k-j}
{{i+j} \choose {2j+1}}$$
and
$$u_{k+1}=(-1)^k {{\frac{q+1}{2}+k}\choose {2k+1}}+ \sum
_{j=0}^{k}(-1)^j u_{k-j} {{\frac{q+1}{2}+j} \choose {2j+2}}.$$

\begin{Lemma} \label{Bl-Cl}
We have $\gamma_N(0)=0$, $\gamma_N(i)>0$ for $i=1,...,q-1$, and
$\gamma _N(q-i)=\gamma _N(i)$. Moreover, for $i=1,..., \frac{q-1}{2}$ we have
$$\gamma_{N}(i)=\left\{
\begin{array}{lc}
F_k(i)&\hbox{if $N=2k+2$,}\\
G_k(i)&\hbox{if $N=2k+1$.} \\
\end{array}
\right.$$
\end{Lemma}

\begin{proof}
Using the same method as in the proof of Lemma
\ref{p^n} we get the following recurssion:
$$2\gamma_N (i)=\sum _{|j|\le \frac{q-1}{2}}\gamma _{n}(i+j).$$
Then one can see that
$$\gamma _1(i)=\left\{ \begin{array}{cl}
0&\hbox{if $i=0$,}\\
1&\hbox{if $i=1,\dots , q-1$,}\\
\end{array}\right.$$
$\gamma _N(0)=0$, $\gamma _N(q-i)=\gamma _N(i)$ for $i=1,..., q-1$ and
$$\gamma _N(i)=\sum _{j=1}^i\gamma _n\left( \frac{q+1}{2} -j\right)$$
for $i=1,..., \frac{q-1}{2}.$ Using this recurrsion and the formula
$$\sum _{j=k}^n{j\choose k}={{n+1}\choose {k+1}},$$ one can easily
prove the lemma by induction.
\end{proof}

\medskip

Set $B=A/(x_0^q)$ and $C=A/(I:x_0^q)$. Note that $A$ and $C$ are
graded $0$-dimensional Gorenstein rings (in fact, $A$ is a complete
intersection ring).
From the definition of $A$ one can easily see that $\dim A_i= 0$ if
and only if $i<0$ or $i>N(q-1)+1$. As $C_i\subset A_{i+q}$, this implies that $C_i=0$ if $i<0$
or $i>N(q-1)+1-q=n(q-1).$

From now on we consider the case $s=1$, i.e., $q=p$.

\begin{Proposition} \label{C=B}
If $p>2$ then $C_d=B_{d}$ for $d\le d_N=\frac{1}{2}n(p-1).$
\end{Proposition}

\begin{proof}
Since $x_0^{2p}=(-\sum_{i\ge 1}x_i^2)^p$ in $S$, we see that
$I+(x_0^p)\subset I+(I:x_0^p)$. Our assertion is equivalent to the fact
that this inclusion is  equality in gradings $d \le d_N$.
Therefore it is sufficient to show that if $g_0\in (I:x_0^p)_d$, $d\le d_N$
then $g_0 \in I+x_0^p$.

Assume that $g_0\in (I:x_0^p)_d$ for some $d\le d_N$. Then there exist
some homogeneous polynomials $g_i$, $i=1,\dots, N$ and $h$, such that
$$g_0x_0^p=\sum _{i=1}^Ng_ix_i^p+(\sum _{i=0}^N x_i^2) h.$$
Then by Proposition \ref{diff} we can write $h$ as
$$h=\sum _{i=0}^N x_i^p h_i+(\sum _{i=0}^N x_i^2)^{p-1}h'$$
for some homogeneous polynomials $h_i$, $i=0,...,N$ and $h'$.
Then
$$g_0x_0^p=\sum _{i=1}^N(g_i+(\sum _{i=0}^N x_i^2)h_i)x_i^p+(\sum
_{i=0}^N x_i^2)h_0x_0^p + (\sum _{i=0}^N x_i^{2p})h'$$ and hence
$$(g_0-x_0^ph_0-(\sum _{i=0}^N x_i^2)h')x_0^p=\sum _{i=1}^N(g_i+(\sum
_{i=0}^N x_i^2)h_i+x_i^ph')x_i^p.$$ Now comparing coefficients of both
sides treated as polynomials in $x_0$ we see that $g_0-x_0^ph_0-(\sum
_{i=0}^N x_i^2)h'\in (x_1^p, \dots , x_N^p)$, which finishes the
proof.
\end{proof}

\medskip

The above proposition allows us to compute the Hilbert functions of
$B$ and $C$:
\begin{Lemma} \label{Hilb-B}
$$\dim {B_{i}}=\left\{
\begin{array}{lc}
\sum_{j\ge 0}(-1)^j\dim A_{i-jp} &\hbox{if $i\le d_N+p$,}\\
\dim B_{2d_N-i}&\hbox{if $i\ge d_N+p$.}\\
 \end{array}\right.$$
In particular, $B_i=C_i$ if $i\le d_N$ or $i\ge d_N+p$. We also have
$\dim C_{2d_N-i}=\dim C_i$ for all integers $i$.  Moreover, $B_i\ne 0$
($C_i\ne 0$) if and only if $0\le i\le 2d_N=n(p-1)$.
\end{Lemma}

\begin{proof}
Since we have short exact sequences
$$0\to C_{i}\mathop{\to}^{x_0^p} A_{i+p}\to B_{i+p}\to 0,$$
Proposition \ref{C=B} implies that
$$\dim {B_{i+p}}-\dim {A_{i+p}}= -\dim C_i= -\dim B_i $$ for $i\le
d_N$. This allows us to compute $\dim B_i$ for $i\le d_N+p$ and $\dim
C_i$ for $i\le d_N$.  An easy calculation implies that $\dim C_i$ is
increasing for $i\le d_N$. Let us recall that $C$ is Gorenstein and
hence the Hilbert function of $C$ is symmetric.  Hence the above
remark together with $C_{2d_N+1}=0$ imply that $\dim C_i=\dim
C_{2d_N-i}$.

Assume that $i\ge d_N$. Then using once more the above short exact
sequence we get
$$\dim {B_{i+p}}-\dim {A_{i+p}}= -\dim C_i= -\dim B_{2d_N-i}.$$
Finally, we can use the duality of the graded Gorenstein ring $A$ to
get
$$\dim B_{i+p}= \dim A_{i+p}-\dim B_{2d_N-i}=\dim A_{2d_N-i}-\dim
B_{2d_N-i}=\dim B_{2d_N-i-p}.$$
The remaining part of the lemma follows from Proposition \ref{C=B} and
equality of dimensions $\dim B_i=\dim C_i$ for $i\ge d_N+p$.
\end{proof}

\medskip

\begin{Lemma}\label{sum-B}
Let $l$ be an integer. Then
$$\sum _{i\in \ZZ} \dim B_{l+ip}=p^n+2^n\gamma_{N}(l_0),$$
where $l_0$ is the unique integer such that
$0\le l_0<p$ and $l\equiv d_N+l_0 \pmod{p}$,
and  $\gamma _N (\cdot )$ is as in Lemma  \ref{Bl-Cl}.
\end{Lemma}

\begin{proof}  Let us set $l_1=d_N+l_0$.
By Lemma \ref{Hilb-B}
we have
\begin{eqnarray*}
&\sum _{i\in \ZZ } \dim A_{l+i p}= \sum _{i\in \ZZ} (\dim B_{l+ip}+
\dim C_{l+ip-p})= 2 \sum _{i\in  \ZZ} \dim B_{l+ip}
-(\dim B_{l_1}-\dim C_{l_1}).\\
\end{eqnarray*}
But
\begin{eqnarray*}
&\dim B_{l_1}-\dim C_{l_1}=\sum_{j\ge 0}(-1)^j\dim
A_{l_1-jp}+\sum_{j\ge 0}(-1)^{j+1}\dim A_{2d_N-l_1-jp}\\ &=\sum_{j\ge
0}(-1)^j\dim A_{l_1-jp}+\sum_{j\le -1}(-1)^{j}\dim A_{l_1-jp}=
\sum_{j\in \ZZ}(-1)^j\dim A_{l_1-jp}.\\
\end{eqnarray*}
Therefore by Lemma \ref{p^n}
$$\sum _{i\in \ZZ} \dim B_{l+ip}=p^n+\frac{1}{2}\sum_{j\in \ZZ}(-1)^j\dim A_{l_1-jp},$$
which together with Lemma \ref{Bl-Cl} proves the required equality.
\end{proof}

\section{Vanishing and non-vanishing theorems}

In this section we prove some basic vanishing and non-vanishing
theorems for cohomology of twisted Frobenius pull-backs of spinor
bundles.

Let us set $\psi_1= \Omega_{\PP^N}(1)|_{Q_n}$.

\begin{Proposition} \label{psixspin}
For any spinor bundle $\Sigma$ on $Q_n$ we have
$$h^1(Q_n, \psi_1\otimes \Sigma (t))=\left\{
\begin{array}{cl}
0\quad&\hbox{if $t\ne 0$,}\\
2\rk \Sigma=2^{{[\frac{N}{2}]}}\quad&\hbox{if $t= 0$.}\\
\end{array}
\right.$$
\end{Proposition}

\begin{proof}
Every spinor bundle $\Sigma$ is ACM and so it fits into the following
short exact sequence of sheaves on $\PP^N$
\begin{equation}
0\to \O_{\PP^N}^{2^{[N/2]}}(-1)\to \O_{\PP^N}^{2^{[N/2]}}\to
i_*\Sigma(1) \to 0,\label{(*)}
\end{equation}
where $i: Q_n\hookrightarrow \PP^N$ is the embedding (see (\ref{(i')}), (\ref{(ii')}) and (\ref{(iii')})).
Tensoring this sequence with $\Omega^1_{\PP^N}(t)$ and using standard Bott formulas
for cohomology of twists of $\Omega^1_{\PP^N}$ on $\PP^N$ we get the
result.
\end{proof}

\medskip

\begin{Corollary}
If $E$ is arithmetically Cohen-Macaulay on $Q_n$ then it is a direct
sum of line bundles if and only if
$$\sum _{t\in \ZZ} h^1(E\otimes \psi_1(t))=\rk E.$$
\end{Corollary}

\begin{proof}
By Theorem \ref{ACM-Q} any ACM bundle on $Q_n$ is isomorphic to a
direct sum of line bundles $\O_{Q_n}(i)$ and spinor bundles twisted by
some $\O _{Q_n}(i)$. By Proposition \ref{psixspin} we see that
$$\sum _{t\in \ZZ} h^1(\Sigma (i) \otimes \psi_1(t))=2 \rk \Sigma.$$
On the other hand
$$h^1(Q_n, \psi_1 (t))=\left\{
\begin{array}{cl}
0\quad&\hbox{if $t\ne -1$,}\\
1\quad&\hbox{if $t=-1$,}\\
\end{array}
\right. $$ so if $\sum _{t\in \ZZ} h^1(E\otimes \psi_1(t))=\rk E$
then $E$ cannot contain any twists of spinor bundles as direct summands.
\end{proof}

\medskip

Let us recall that $d_N=(N-1)\frac{p-1}{2}=n\frac{p-1}{2}$.

\begin{Corollary} \label{dir-sum-lb}
$F_*(\O_{Q_n}(t))$ is a direct sum of line bundles if and only if
$t-d_{N}$ is divisible by $p$.
\end{Corollary}

\begin{proof}
By \cite[Theorem 4]{BM} the graded $S$-module $\bigoplus _{t\in \ZZ}
H^1(Q_n, F^*\psi_1(t-p))$ is isomorphic as a graded $S$-module to
$B$. But $H^1(Q_n, F^*\psi_1(t+ip))=H^1(Q_n, F_*(\O_{Q_n}(t))\otimes
\psi_1(i)),$ so by the above corollary the fact that
$F_*(\O_{Q_n}(t))$ is a direct sum of line bundles is equivalent to
equality
$$\sum _{i\in \ZZ} \dim B_{t+ip}=p^n.$$
Now the required assertion follows from Lemma \ref{sum-B}.
\end{proof}

The following vanishing theorem allows to compute the decomposition of
Frobenius push-forwards of line bundles:

\begin{Theorem} \label{vanishing}
Let $\Sigma$ be a spinor bundle on $Q_n$. Then for $0<i<n$ we have
$H^i(Q_n, F^*\Sigma (t)) =0$ if $t\le d_N-ip$ or $t\ge d_N-(i-1)p$.
\end{Theorem}

\begin{proof}
By Lemma \ref{spin-duality} it is sufficient to prove vanishing of
$H^i(Q_n, F^*\Sigma (t))$ for $t\ge d_N-(i-1)p$. By Corollary
\ref{dir-sum-lb} and the projection formula we have
$$H^i(Q_n, F^*\Sigma (d_N+tp))=H^i(Q_n, F_*(\O_{Q_n}(d_N))\otimes \Sigma
(t))=0$$ for any integer $t$. In particular, $H^i(Q_n, F^*\Sigma
(d_N-(i-1)p))=0$.  Now the proof is by induction on $n$.

For $n=2$, $\Sigma=\O_{\PP^1 \times \PP^1}(-1,0)$ or
$\Sigma=\O_{\PP^1\times \PP^1}(0,-1)$ and in both cases it is easy
to check the required assertion. Assume that the theorem holds for
quadrics of dimension less than $n$.
Let us recall that the restriction of the spin representation of
$\Spin (2m+1)$ to $\Spin (2m)$ is the sum of the two spin
representations of $\Spin (2m)$.  Similarly, the restriction of either
spin representation of $\Spin (2m)$ to $\Spin (2m-1)$ is the spin
representation. Therefore the restriction of a spinor bundle to a
hypersurface quadric is either a spinor bundle or a direct sum of two
spinor bundles.  Using the long cohomology sequence for the short
exact sequence
$$0\to F^*\Sigma (t)\to F^*\Sigma (t+1)\to F^*\Sigma (t+1)|_{Q_{n-1}}\to 0$$
and the induction assumption we see that for $t\ge d_n-(i-1)p-1$ we have
a surjection
$$H^i(Q_n, F^*\Sigma (t))\to H^i(Q_n, F^*\Sigma (t+1))\to 0.$$
In particular, vanishing of $H^i(Q_n, F^*\Sigma (d_N-(i-1)p))$ implies vanishing
of $H^i(Q_n, F^*\Sigma (t))$ for $t\ge d_N-(i-1)p$.
\end{proof}

The above vanishing theorem implies that the Frobenius pull-back of
the spinor bundle on $Q_3$ is very similar to an instanton
bundle. More precisely, we have the following proposition:

\begin{Proposition}
Let us set $E=F^*\Sigma(\frac{p-1}{2})$ on $Q_3$. Then $E$ is the
cohomology of the monad
$$ 0\to \O (-1)^b\to \Sigma ^{b+1}\to \O ^b\to 0,$$
where $b=h^1(E)=-\chi (E).$
\end{Proposition}

\begin{proof}
The proof is an application of Horrock's killing technique and it is
quite similar to the proof of \cite[Proposition 1.1]{OS}.
We leave the details to the reader.
\end{proof}

\begin{Corollary} \label{non-vanishing}
Let $\Sigma$ be a spinor bundle on $Q_n$. Then for $0<i<n$ we have
$H^i(Q_n, F^*\Sigma (t)) \ne 0$ if $d_N-ip <t< d_N-(i-1)p$.
\end{Corollary}

\begin{proof}
If $H^i(Q_n, F^*\Sigma (t)) = 0$ for some $d_N-ip <t< d_N-(i-1)p$, then
by  Theorem \ref{vanishing}
$$H^i(Q_n, F_*(\O_{Q_n} (t))\otimes \Sigma (j))=H^i(Q_n, F^*\Sigma
(t+jp)) = 0$$ for all integers $j$.

Let us note that if $n$ is even then there exists an automorphism $a:
Q_n\to Q_n$ such that $a^*\Sigma_{\pm}\simeq \Sigma_{\mp}$. Therefore
cohomology groups of $F^*\Sigma_{+}(t+jp)$ and $F^*\Sigma_{-}(t+jp)$ are the
same.  In particular, the above vanishing holds for all spinor bundles
on $Q_n$ and we can apply Corollary \ref{dir-sum1}.  But then we get
contradiction with Corollary \ref{dir-sum-lb}.
\end{proof}

\begin{Theorem} \label{2-vanishing}
Let $\Sigma_1, \Sigma_2$ be spinor bundles on $Q_{n}$, $n\ge 2$.  Then
for any $0<i<n$ we have $H^i(Q_n, \Sigma _1\otimes F^*\Sigma _2(t))=0$ if
$t\le d_N-ip$ or $ t\ge d_N-(i-1)p+1$.
\end{Theorem}

\begin{proof}
For simplicity of notation let us consider only odd dimensional
quadrics $Q_n$, $n=2m+1$ (the proof in the even dimensional case is
essentially the same). As before we can easily reduce to the case
$i=1$ (the proof of Lemma \ref{2-spin-duality} gives vanishing of
higher intermediate cohomology groups). Let us note the following
short exact sequences:
\begin{equation}\label{seq1}
0\to F^*\Sigma (t-p)\to \O _{Q_n} (t-p)^{2^{m+1}}\to F^*\Sigma (t)\to 0
\end{equation}
and
\begin{equation}\label{seq2}
0\to \Sigma (t-p)\to \O _{Q_n} (t-p)^{2^{m+1}}\to \Sigma (t-p+1)\to 0.
\end{equation}
 Using the long cohomology sequnces for appropriate twists we get
the following exact sequences:
$$0\to H^0(\Sigma \otimes F^*\Sigma (t-p))\to H^0(\Sigma (t-p))^{\oplus 2^{m+1}}\to
H^0(\Sigma \otimes F^*\Sigma (t))\to H^1(\Sigma \otimes F^*\Sigma (t-p)) \to 0
$$
for all $t$, and by Theorem \ref{vanishing}
\begin{equation}\label{seq3}
0\to H^0(\Sigma \otimes F^*\Sigma (t-p))\to H^0(F^*\Sigma (t-p))^{\oplus 2^{m+1}} \to
H^0(\Sigma \otimes F^*\Sigma (t-p+1))\to H^1(\Sigma \otimes F^*\Sigma (t-p)) \to 0
\end{equation}
for $t\le d_N$.  Using these sequences we get the following
recurrence equation
\begin{equation}\label{recur}
h^0(\Sigma \otimes F^*\Sigma (t))=2^{m+1}(h^0(\Sigma
(t-p))-h^0(F^*\Sigma (t-p))) +h^0(\Sigma \otimes F^*\Sigma (t-p+1))
\end{equation}
for $t\le d_N$. Let us first prove that
\begin{equation}\label{seq4}
h^0(\Sigma \otimes F^*\Sigma (t))+h^0(\Sigma \otimes F^*\Sigma
(t+1))=2^{m+1}h^0(F^*\Sigma (t)) \label{h+h=h}
\end{equation}
for $t\le d_N-1$.  We prove it by induction on $t$ starting with very
negative $t$ for which the equality is obvious. By (\ref{recur}) and
the induction assumption for $t\le d_N-1$ the left hand side of (\ref{h+h=h}) is equal
to
$$
\begin{array}{rl}
&2^{m+1}(h^0(\Sigma (t-p))-h^0(F^*\Sigma
(t-p))+h^0(\Sigma(t-p+1))-h^0(F^*\Sigma (t-p+1)))\\ &+h^0(\Sigma
\otimes F^*\Sigma (t-p+1))+h^0(\Sigma \otimes F^*\Sigma (t-p+2))=\\
=&2^{m+1}(h^0(\Sigma (t-p))-h^0(F^*\Sigma (t-p))+h^0(\Sigma
(t-p+1))).\\
\end{array}
$$
But sequences (\ref{seq1}) and (\ref{seq2}) imply that for $t\le d_N$
$$  \begin{array}{rl}
h^0(F^*\Sigma (t-p))+h^0(F^*\Sigma (t))&=2^{m+1} h^0(\O_{Q_n}(t-p))=
h^0(\Sigma (t-p))+h^0(\Sigma (t-p+1)),\\
\end{array}$$
which proves the required assertion.

Now let us note that (\ref{seq3}) and (\ref{seq4}) imply vanishing of
$H^1(\Sigma \otimes F^*\Sigma (t))$ for $t\le d_N-p$.  By Lemma
\ref{2-spin-duality} this implies vanishing $H^1(\Sigma \otimes
F^*\Sigma (t))$ for $t\ge d_N+1$.
\end{proof}

\begin{Corollary} \label{2-non-vanishing}
Let us set $S_n=\Sigma$ if $n$ is odd and $S_n=\Sigma_{+}\oplus \Sigma_{-}$ if
$n$ is even.  Then for any $0<i<n$ we have $H^i(Q_n, S_n \otimes F^*S_n
(t))\ne 0$ if $d_N-ip+1\le t\le d_N-(i-1)p$.
\end{Corollary}

\begin{proof}
As before it is sufficient to prove the corollary for $i=1$.  The
proof is by induction on $n$. For $n=2$ the group $H^1(S_2 \otimes
F^*S_2(t))$ contains $H^1(\O_{{\PP^1}\times {\PP^1}}(t, t-p-1))$ as a
direct summand and this cohomology group is non-zero if $0\le t\le
p-1$.

Assume we know the statement in dimensions less than $n$.  Then we
have an exact sequence
$$H^1(S_n\otimes F^*S_n(t-1))\to H^1(S_n\otimes F^*S_n(t))\to
H^1(S_n\otimes F^*S_n(t)|_{Q_{n-1}})\to H^2(S_n\otimes F^*S_n(t-1)).$$
This sequence and Theorem \ref{2-vanishing} imply that
$h^1(S_n\otimes F^*S_n(t-1)) \le h^1(S_n\otimes F^*S_n(t))$
if $t\ge d_{n}+1$, so by Lemma \ref{2-spin-duality} it is sufficient to prove
that $H^1(S_n\otimes F^*S_n(d_N))\ne 0$.
But Lemma \ref{2-spin-duality} implies also that
$$H^2(S_n\otimes F^*S_n(d_N-p))\simeq H^1(S_n\otimes F^*S_n(d_N))
\simeq H^1(S_n\otimes F^*S_n(d_N-p+1))^*,$$ so the above sequence
applied for $t=d_N-p+1$ gives by the induction assumption
non-vanishing of $H^1(S_n\otimes F^*S_n(d_N))$.
\end{proof}

\section{Decomposition of Frobenius push-forwards}

If $E$ is an ACM bundle then $F_*E$ is also an ACM bundle. In
particular, if $n\ge 3$ then Frobenius push-forwards of line bundles
and twisted spinor bundles split into a direct sum of line bundles and
twisted spinor bundles. In this section we study the corresponding
decompositions.

Let $S_n$ be as in Corollary \ref{2-non-vanishing}. Let $q=p^s$ for
some non-negative integer $s$.

\begin{Proposition} \label{splitting}
\begin{enumerate}
\item If $F^s_*(\O_{Q_n}(j))$ contains $\O_{Q_n}(-t)$ as a direct summand then
$0\le tq+j\le n(q-1)$.
\item If $F^s_*(S_n(j))$ contains $\O_{Q_n} (-t)$ as a direct summand
then  $1\le tq+j\le n(q-1)$.
\end{enumerate}
\end{Proposition}

\begin{proof}

(1) If $F^s_*(\O_{Q_n}(j))$ contains $\O_{Q_n}(-t)$ as a direct summand then
$$0\ne H^0(\O_{Q_n})\subset H^0(F^s_*(\O_{Q_n}(j))\otimes
\O_{Q_n}(t))=H^0(\O_{Q_n}(j+tq)),$$ which implies that $tq+j\ge
0$. Similarly,
$$0\ne H^n(\omega_{Q_n})\subset H^n(F^s_*(\O_{Q_n}(j))\otimes
\O_{Q_n}(t-n))=H^n(\O_{Q_n}(j+(t-n)q))$$ which implies that $tq+j\le
n(q-1)$.

(2) If $F^s_*(S_n(j))$ contains $\O_{Q_n}(-t)$ as a direct summand then
$$0\ne H^0(\O_{Q_n})\subset H^0(F^s_*(S_n(j))\otimes \O_{Q_n}(t))=H^0(S_n(j+tq)),$$
which implies that $tq+j\ge 1$. Similarly, by the Serre duality
$$0\ne H^n(\omega_{Q_n})\subset H^n(F^s_*(S_n(j))\otimes
\O_{Q_n}(t-n))=H^n(S_n(j+(t-n)q))=(H^0(S_n^*(n(q-1)-tq-j)))^*$$
which implies that $tq+j\le n(q-1)$.
\end{proof}

\begin{Proposition} \label{splitting2}
\begin{enumerate}
\item $F_*(\O_{Q_n}(j))$ contains $\Sigma(-t)$ as a direct summand if and
only if $d_N-p+1\le tp+j\le d_N-1$.  In this case, $F_*(\O_{Q_n}(j))$
contains $S_n(-t)$.  In particular,  $F_*(\O_{Q_n}(j))$ contains
at most one twist of a spinor bundle.
\item $F_*(S_n(j))$ contains $\Sigma (-t)$ as a direct summand if and
only if $d_N-p+1\le tp+j\le d_N$. In this case, $F_*(S_n(j))$
contains also $S_n(-t)$.  In particular,  $F_*(S_n(j))$ contains
exactly one twist of a spinor bundle.
\end{enumerate}
\end{Proposition}

\begin{proof}

(1) If $F_*(\O_{Q_n}(j))$ contains $\Sigma(-t)$ as a direct summand
then by symmetry (see the proof of Corollary \ref{non-vanishing}) it
also contains $S_n(-t)$. By Lemma \ref{tensor-spin} this happens if
and only if $H^1(Q_n, F^*\Sigma (j+tp))=H^1(F_*(\O_{Q_n}(j)) \otimes
\Sigma (t))\ne 0$ so the assertion follows from Theorem
\ref{vanishing} and Corollary \ref{non-vanishing}.

(2) If $F_*(S_n(j))$ contains $\Sigma(-t)$ as a direct summand
then by symmetry it also contains $S_n(-t)$. By Lemma
\ref{tensor-spin} this happens if and only if $H^1(Q_n, S_n\otimes F^*S_n
(tp+j))=H^1(F_*(S_n(j)) \otimes S_n (t))\ne 0$ so the assertion follows from
Theorem \ref{2-vanishing} and Corollary \ref{2-non-vanishing}.
\end{proof}

\begin{Corollary}\label{quasi-exc}
For any line bundle $\L$ on $Q_n$, $n\ge 3$, the bundle $F_*\L$ is
quasi-exceptional.
\end{Corollary}

\begin{proof}
The assertion follows from Lemma \ref{tensor-spin} and Propositions \ref{splitting}
and \ref{splitting2}.
\end{proof}

\medskip

Let us fix an integer $0\le j<p$. By Propositions \ref{splitting} and
\ref{splitting2} we can write $$F_*(\O_{Q_n}(d_N+j))=\bigoplus \O
_{Q_n}(-t)^{a_t}\oplus S_n (1)^b,$$ where $b=0$ if $j=0$ and $a_t=0$
if $|tp+j|>d_N$.

\begin{Theorem} \label{decomposition}
If $|tp+j|\le d_N$ then $0<a_t=\dim C_{d_N+tp+j}$.
Moreover,  $b=2^{[N/2]}\gamma _N(j)$, where $\gamma _N(\cdot )$ is as in
Lemma \ref{Bl-Cl}.
\end{Theorem}

\begin{proof} By Proposition \ref{psixspin} we have
$$h^1(F_*(\O_{Q_n}(d_N+j))\otimes \psi_1(t-1))=\left\{
\begin{array}{cl}
a_{0}+2^{[n/2]+1}b&\quad \hbox{ if $t=0$,}\\
a_t&\quad \hbox{ if $t\ne 0$.}\\
\end{array}
\right.$$
On the other hand, we have
$$h^1(F_*(\O_{Q_n}(d_N+j))\otimes
\psi_1(t-1))=h^1(F^*\psi_1(d_N+(t-1)p+j)) =\dim B_{d_N+tp+j},$$ so
$a_t=\dim B_{d_N+tp+j}=\dim C_{d_N+tp+j}$ for $t\ne 0$.  Comparing
ranks in the decomposition we get $\sum _{t\in \ZZ}
a_t+2^{[n/2]}b=p^n$. Therefore
$$\sum \dim B_{d_N+ip+j}=\sum _{t\in \ZZ}
a_t+2^{[n/2]+1}b=2^{[n/2]}b+p^n$$ and Lemma \ref{sum-B} implies that
$b=2^{[N/2]}\gamma _N(j)$. By the proof of Lemma \ref{sum-B}
$$a_0=\dim B_{d_N+j}-2^N\gamma _N(j)=\dim B_{d_N+j}- (\dim
B_{d_N+j}-\dim C_{d_N+j})=\dim C_{d_N+j},$$
which finishes the proof.
\end{proof}

\medskip

Since $F_*(\O_{Q_n}(j+tp))\simeq F_*(\O_{Q_n}(j))\otimes \O_{Q_n}(t)$,
the above theorem gives the decomposition for all Frobenius push
forwards of line bundles on $Q_n$ ($n\ge 3$).

We can also compute the decomposition of $F_*(S_n(j))$ along the
following lines.  By (\ref{seq2}) and (\ref{seq4}) we have an exact
sequence
$$
0\to H^1(S_n\otimes F^*S_n (t))\to H^1(F^*S_n (t))^{\oplus 2^{[N/2]}} \to
H^1(\Sigma \otimes F^*S_n (t+1))\to H^2(S_n \otimes S_n (t))
$$
for $t\le d_N-1$. By Theorem \ref{2-vanishing} the last cohomology
group vanishes for $t\ge d_N-p+1$. By the same theorem and the proof
of Corollary \ref{2-non-vanishing} we also have an exact sequence
$$0\to H^1(S_n\otimes F^*S_n(d_N-p+1))\to H^1(S_n\otimes
F^*S_n(d_N-p+1)|_{Q_{n-1}})\to H^1(S_n\otimes F^*S_n(d_N-p+1))^*\to
0.$$ Together with Lemmas \ref{tensor-spin} and \ref{2-spin-duality}
this is sufficient to determine the required decomposition.  By
induction, this also gives decomposition of $F^s_*(\O_{Q_n}(j))$.  We
skip the actual computation as it is long and it will not be used in
the following.

\begin{Corollary} \label{strong-splitting}
$F^s_*(\O_{Q_n}(j))$ contains $\O_{Q_n}(-t)$ as a direct summand if
and only if $0\le tq+j\le n(q-1)$.
\end{Corollary}

\begin{proof}
By Proposition \ref{splitting} we need only to show that some line
bundles appear in the decomposition. The proof is by induction on
$m$. For $s=1$ the required assertion follows from Theorem
\ref{decomposition}.  Assume that $0\le tp^{s+1}+j\le n(p^{s+1}-1)$
for some $t$.  Then $0\le tp^{s}+j/p\le n(p^{s}-1)+n-n/p$. Therefore
there exist an integer $l$ such that $0\le tp^{s}-l\le n(p^{s}-1)$ and
$-j/p\le l\le n-(n+j)/p$.  Then $\O_{Q_n}(-t)$ is a direct summand of
$(F^m)_*(\O_{Q_n}(-l))$ and $\O_{Q_n}(-l)$ is a direct summand of
$F_*(\O_{Q_n}(j))$.
\end{proof}

\begin{Corollary}\label{t1}
$F_*\O_{Q_n}$ is a tilting bundle if and only if $p> n$.
\end{Corollary}

\begin{proof}
By Theorem \ref{decomposition} $F_*\O_{Q_n}$ contains $\O _{Q_n}(-i)$
as a direct summand if and only if $0\le ip\le n(p-1)$. Moreover,
$F_*\O_{Q_n}$ contains at most one twist of $S_n$. If $p< n$ then
$i\le n-2$ as $n(p-1)<p(n-1)$. This implies that $F_*\O_{Q_n}$
Karoubian generates a proper subcategory of $D^b(Q_n)$ generated by at
most $(n-1)$ line bundles and one twist of spinor bundles.  If $p=n$
then $F_*\O_{Q_n}$ is a direct sum of line bundles so it does not
generate $D^b(Q_n)$. On the other hand, if $p> n$ then $F_*\O_{Q_n}$
contains as direct summands $\O_{Q_n},\dots , \O_{Q_n}(-n+1)$ and one
twist of $S_n$, so it is tilting.
\end{proof}

\begin{Corollary}\label{t2}
If $n=2m$, $m\ge 2$ and $s\ge 2$ then $F^s_*\O_{Q_n}$ is
quasi-exceptional only if $m=2$, $p=3$ and $s=2$. In this case
$F^s_*\O_{Q_n}$ is also tilting.
\end{Corollary}

\begin{proof}
First, let us assume that $m>2$ or $p>3$. Note that $F_*\O_{Q_n}$ or
$F_*\O_{Q_n}(-1)$ contain $S_n(-m+\lceil m/p \rceil)$ as a direct summand. Set
$l_0=[n-n/p]$. By assumption $l_0>m$ so $F_*\O_{Q_n}(-l_0)$ contains
$S_n(-m)$ as a direct summand.  Note that $F_*\O_{Q_n}$ contains
$\O_{Q_n}$, $\O_{Q_n}(-1)$ and $\O_{Q_n}(-l_0)$ as direct summands, so
$F^2_*\O_{Q_n}$ contains as direct summands both $S_n(-m+\lceil m/p\rceil)$ and
$S_n(-m)$. Then Lemma \ref{tensor-spin} implies that $F^2_*\O_{Q_n}$ is
not quasi-exceptional. Since $F^s_*\O_{Q_n}$ contains as a direct
summand $F^2_*\O_{Q_n}$ it is also not quasi-exceptional.

Now assume that $m=2$ and $p=3$. Then $F^2_*\O_{Q_4}$ contains as
direct summands only $\O_{Q_4},\dots , \O_{Q_4}(-3)$ and $S_4(-1)$.
But $F_*(\O_{Q_4}(-3))$ contains $S_4(-2)$ as a direct summand and hence
$F^s_*\O_{Q_4}$ is not quasi-exceptional for $s\ge 3$.
\end{proof}

\begin{Corollary} \label{t3}
Assume that $n=2m+1$ and $s\ge 2$. If $p\ge n$ then $F^s_*\O_{Q_n}$
is a tilting bundle. If $p<n$ then $F^s_*\O_{Q_n}$ is not
quasi-exceptional.
\end{Corollary}

\begin{proof}
It is easy to see that if $p\ge n$ then $F_*(\O_{Q_n}(-j))$ for
$0\le j\le 2m$ contain as direct summands only line bundles and
$\Sigma (-m)$. Moreover, $F_*(\Sigma (-m))$ contains as direct
summands only line bundles and $\Sigma (-m)$. This allows easily to
check that $F^s_*\O_{Q_n}$ is a tilting bundle.\

Now assume that $p<n$. Let us note that by Corollary
\ref{strong-splitting} $F_*\O_{Q_n}$ always contains $\O_{Q_n}(-t)$
for $0\le t\le m$, as $0\le mp\le n(p-1)=mp+(m+1)(p-2)+1$. But by
Proposition \ref{splitting2} $F_*(\O_{Q_n}(-m))$ contains
$\Sigma (-m)$ as a direct summand and $F_*(\O_{Q_n}(-m+\frac
{p+1}{2}))$ contains $\Sigma (-m+1)$ as a direct summand.  This
implies that $F^2_*\O_{Q_n}$, and hence $F^s_*\O_{Q_n}$, are not
quasi-exceptional.
\end{proof}

\bigskip

{\bf Acknowledgements.}

The author would like to thank A. Samokhin for providing the original
motivation for this paper by writing \cite{Sa1} and for stimulating
discussions.

The paper was written while the author's stay in the Max Planck
Institute f\"ur Mathematik in Bonn. The author would like to thank
this institution for the hospitality.

In the preparation of this paper the author used the computer algebra
system {\sc Singular} by G.-M. Greuel, G. Pfister, and H. Sch\"onemann
to check some results (e.g., Theorem \ref{vanishing} in dimensions
$3,4,5$ and in low characteristic) and in some cases to guess the
results (see Proposition \ref{C=B}).  All traces of actual computer
calculations were carefully removed from the paper.

Finally, the author would like to thank his student, Piotr
Achinger, who pointed out a gap in the previous proof of
Proposition \ref{diff-new}.

\end{document}